\documentclass{amsart}
\usepackage{amsmath,amsfonts,amsthm}
\usepackage{amssymb,latexsym}
\usepackage{graphics}
\usepackage[colorlinks]{hyperref}

\usepackage[all]{xypic}

\theoremstyle{plain}
\theoremstyle{definition}
\newtheorem{theorem}{Theorem}[section]
\newtheorem{lemma}[theorem]{Lemma}

\newtheorem{corollary}[theorem]{Corollary}

\newtheorem{note}[theorem]{Note}

\newtheorem{remark}[theorem]{Remark}
\theoremstyle{remark}

\numberwithin{equation}{section}

\title{Generalized shifts through derivations' concept in $\ell^p(\tau)$ spaces}
\author[S. Arzanesh, F. Ayatollah Zadeh Shirazi, A. Hosseini]{Safoura Arzanesh, Fatemah Ayatollah Zadeh Shirazi, Arezoo Hosseini}
\begin{document}
\begin{abstract}
In the following text for $p\in[1,\infty]$, nonzero cardinal number $\tau$, self--map $\varphi:\tau\to\tau$ if
there exists $N\in\mathbb{N}$ such that $\varphi^{-1}(\alpha)$ has at most $N$ elements for each $\alpha<\tau$,
and operators $\psi,\lambda:\ell^p\tau)\to\ell^p(\tau)$ we prove the generalized shift
$\mathop{\sigma_\varphi\restriction_{\ell^p(\tau)}:\ell^p(\tau)\to\ell^p(\tau)\:\:\:\:\:\:\:\:\:}\limits_{\:\:\:\:\:\:\:\:\:
(x_\alpha)_{\alpha<\tau}\mapsto
(x_{\varphi(\alpha)})_{\alpha<\tau}}$:
\begin{itemize}
\item is a $(\psi,\lambda)-$derivation if and only if there exists $\mathsf{r}\in{\mathbb C}^\tau$ with
\linebreak
	$\psi={\mathsf r}\sigma_\varphi\restriction_{\ell^p(\tau)}$ and 
	$\lambda=((1)_{\alpha<\tau}-{\mathsf r})\sigma_\varphi\restriction_{\ell^p(\tau)}$,
\item is a $\psi-$derivation if and only if $\psi=\frac12\sigma_\varphi\restriction_{\ell^p(\tau)}$,
\item is not a (Jordan, Jordan triple) derivation.
\item is a generalized (Jordan, Jordan triple)  derivation if and only if $\varphi=id_\tau$.
\end{itemize}
\end{abstract}
\maketitle
\noindent {\small {\bf 2010 Mathematics Subject Classification:}  47B47
\\
{\bf Keywords:}}  Banach algebra, (generalized) derivation, 
generalized shift, 
Jordan derivation, $(\psi,\lambda)-$derivation.
\section{Introduction}
\noindent Just glancing through scientific research engines, one may find a huge amount
of articles in each of the folowing subjects: ``shift theory'' and ``derivations on Banach algebras''.
Our aim in the following text is to have a start point for ``generalized shifts'' in relation with
``derivations'' over $\ell^q(\tau)$ spaces.
\\
Our consideration is in the following sense: for nonzero cardinal number $\tau$,
self--map $\varphi:\tau\to\tau$ consider generalized shift $\sigma_\varphi\restriction_{\ell^q(\tau)}$
over Banach algebra $\ell^q(\tau)$ where $q\in[1,\infty]$ and  $\sigma_\varphi(\ell^q(\tau))\subseteq\ell^q(\tau)$
the following questions arise:
\begin{itemize}
\item Is $\sigma_\varphi\restriction_{\ell^q(\tau)}:\ell^q(\tau)\to\ell^q(\tau)$ is a derivation?
\item Is there any  operator $\psi$ such that $\sigma_\varphi\restriction_{\ell^q(\tau)}:\ell^q(\tau)\to\ell^q(\tau)$ is a $\psi-$derivation?
\item Are there any operator $\psi,\lambda$ such that $\sigma_\varphi\restriction_{\ell^q(\tau)}:\ell^q(\tau)\to\ell^q(\tau)$ is a $(\psi,\lambda)-$derivation?
\item Is $\sigma_\varphi\restriction_{\ell^q(\tau)}:\ell^q(\tau)\to\ell^q(\tau)$ is a generalized derivation?
\item Is there any operator $\psi$ such that $\sigma_\varphi\restriction_{\ell^q(\tau)}:\ell^q(\tau)\to\ell^q(\tau)$ is a generalized $\psi-$derivation?
\item Are there any operator $\psi,\lambda$ such that $\sigma_\varphi\restriction_{\ell^q(\tau)}:\ell^q(\tau)\to\ell^q(\tau)$ is a generalized $(\psi,\lambda)-$derivation?
\end{itemize}
and other related questions.
\\
In the following subsections we bring preliminaries on Banach algebras $\ell^q(\tau)$ and generalized shifts.
Required notes on various definitions of derivation are appeared at the beginning of each section. Most 
authors refer to \cite{sakai} as one of their main references to introduce the subject, we may also
refer the interested reader
to~\cite{survey, kh, nik} for historical background, collections of generalizations, different areas involving with derivations.
\subsection{Background on Banach algebras $\ell^q(\tau)$}
For nonzero cardinal number $\tau$, $x=(x_\alpha)_{\alpha<\tau}\in{\mathbb C}^\tau$ and $q\in(0,\infty]$ let:
\[||x||_q:=\left\{\begin{array}{lc} \left(\mathop{\sum}\limits_{\alpha<\tau}|x_\alpha|^q\right)^\frac1q & q\in(0,\infty) \\
& \\
||x||_\infty=\sup\{|x_\alpha|:\alpha<\tau\} & q=\infty \end{array}\right.\]
and
$\ell^q(\tau)=\{y\in{\mathbb C}^\tau:||y||_q<\infty\}$. For $z=(z_\alpha)_{\alpha<\tau},y=(y_\alpha)_{\alpha<\tau}
\in{\mathbb C}^\tau$ and $p\in(0,\infty)$ we have (for details on $\ell^p(\tau)$ spaces see \cite{folland}):
\begin{eqnarray*}
||zy||_p^p & = & ||(|z_\alpha y_\alpha|^p)_{\alpha<\tau}||_1=||(|z_\alpha|^p)_{\alpha<\tau} (|y_\alpha|^p)_{\alpha<\tau}||_1 \\
& \leq & ||(|z_\alpha|^p)_{\alpha<\tau}||_\infty||(|y_\alpha|^p)_{\alpha<\tau}||_1\:\:(Holder's\: inequality) \\
& = &  ||(|z_\alpha|^p)_{\alpha<\tau}||_\infty ||y||_p^p\leq ||(|z_\alpha|^p)_{\alpha<\tau}||_p ||y||_p^p=||z||_p^p||y||_p^p
\end{eqnarray*}
Hence $||zy||_p\leq||z||_p||y||_p$ also it is evident that $||zy||_\infty\leq||z||_\infty||y||_\infty$.
Note that for $q\in[1,\infty]$, 
$\ell^q(\tau)$ equipped with norm $||\:||_q$, pointwise sum and pointwise product is a Banach algebra.
\subsection{Backgroud on generalized shifts}
For arbitrary set $X$ with at least two elements,
nonempty set $\Gamma$, and self--map $\varphi:\Gamma\to\Gamma$ we call 
$\sigma_\varphi:\mathop{X^\Gamma\to X^\Gamma}\limits_{(x_\alpha)_{\alpha\in\Gamma}\mapsto
(x_{\varphi(\alpha)})_{\alpha\in\Gamma}}$
a generalized shift, which is a generalization of one--sided shift $\mathop{\{1,\ldots,k\}^{\mathbb N}
\to \{1,\ldots,k\}^{\mathbb N}}\limits_{(x_n)_{n\in\mathbb{N}}\mapsto(x_{n+1})_{n\in\mathbb{N}}}$
and two--sided shift $\mathop{\{1,\ldots,k\}^{\mathbb Z}
\to \{1,\ldots,k\}^{\mathbb Z}}\limits_{(x_n)_{n\in\mathbb{Z}}\mapsto(x_{n+1})_{n\in\mathbb{Z}}}$.
Generalized shifts have been introduced for the first
time in \cite{note}, one may find texts in different areas like dynamical systems~\cite{dev, taher},
group theory~\cite{anna}, functional analysis~\cite{compact}, etc. all in generalized shifts' approach.
\\
For nonzero cardinal number $\tau$, self--map $\varphi:\tau\to\tau$ and $q\in[1,\infty]$ we have 
$\ell^q(\tau)\subseteq{\mathbb C}^\tau$ so one may consider $\sigma_{\varphi}\restriction_{\ell^q(\tau)}$.
\subsection{Convention}
In the following text suppose $\tau$ is a nonzero cardinal number, $p\in[1,\infty]$
and $\varphi:\tau\to\tau$ is arbitrary with $\sigma_\varphi(\ell^p(\tau))\subseteq\ell^p(\tau)$. For $E\subseteq \tau$ consider $\chi_E:\tau\to\tau$ with
\[\chi_E(\alpha)=\left\{\begin{array}{lc} 1 & \alpha\in E \\ 0 & {\rm otherwise} \end{array}\right.\]
and ${\mathsf w}^E=(\chi_E(\alpha))_{\alpha<\tau}$. For $\alpha<\tau$ we denote
$\mathsf{w}^{\{\beta\}}$ by ${\mathsf w}^\beta$.
\\
By an operator we mean a continuous linear operator. 
\\
For arbitrary set $X$, $id_X:\mathop{X\to X}\limits_{x\mapsto x}$ is the identity map on $X$.
\begin{remark}
For $p\in[1,\infty)$ the following statements are equivalent~\cite{compact}:
\begin{itemize}
\item $\sigma_\varphi(\ell^p(\tau))\subseteq\ell^p(\tau)$,
\item $\sigma_\varphi(\ell^p(\tau))\subseteq\ell^p(\tau)$ and $\sigma_\varphi\restriction_{\ell^p(\tau)}:\ell^p(\tau))\to\ell^p(\tau)$ is continuous, 
\item there exists $N\in{\mathbb N}$ such that for all $\alpha<\tau$, 
$\varphi^{-1}(\alpha)$ has at most $N$ elements.
\end{itemize}
Moreover, $\sigma_\varphi(\ell^\infty(\tau))\subseteq\ell^\infty(\tau)$ and $\sigma_\varphi\restriction_{\ell^\infty(\tau)}:\ell^\infty(\tau)\to\ell^\infty(\tau)$ is continuous.
\end{remark}
\section{Is $\sigma_\varphi\restriction_{\ell^p(\tau)}$ a $(\psi,\lambda)-$derivation?}
\noindent Let's recall, in Banach algebra $\mathcal A$ \cite{meng}: 
\begin{itemize}
\item for linear mappings $d,\psi,\lambda:\mathcal{A}\to\mathcal{A}$ we say
	$d$ is a $(\psi,\lambda)-$derivation if $d(ab)=d(a)\psi(b)+\lambda(a)d(b)$ for all $a,b\in\mathcal{A}$,
\item for linear mappings $d,\psi:\mathcal{A}\to\mathcal{A}$ we say
	$d$ is a $\psi-$derivation if it is a $(\psi,\psi)-$derivation, i.e.
	$d(ab)=d(a)\psi(b)+\psi(a)d(b)$ for all $a,b\in\mathcal{A}$,
\item we say linear mapping $d:\mathcal{A}\to\mathcal{A}$
	is a derivation if it is an $id_{\mathcal A}-$derivation, i.e.
	$d(ab)=d(a)b+ad(b)$ for all $a,b\in\mathcal{A}$,
\item we say linear mapping $d:\mathcal{A}\to\mathcal{A}$
	is a Jordan derivation if 
	$d(a^2)=d(a)a+ad(a)$ for all $a\in\mathcal{A}$, thus any derivation is a Jordan derivation,
\item we say linear mapping $d:{\mathcal A}\to{\mathcal A}$ is a
	Jordan triple derivation if $d(aba)=d(a)ba+ad(b)a+abd(a)$ for all $a,b\in\mathcal{A}$,
\end{itemize}
In this section we prove that for operators $\psi,\lambda:\ell^p(\tau)\to\ell^p(\tau)$, the generalized shift 
$\sigma_\varphi\restriction_{\ell^p(\tau)}:\ell^p(\tau)\to\ell^p(\tau)$:
\begin{itemize}
\item is a $(\psi,\lambda)-$derivation if and only if 
there exists $\mathsf{r}\in{\mathbb C}^\tau$ with
	$\psi={\mathsf r}\sigma_\varphi\restriction_{\ell^q(\tau)}$ and 
	$\lambda=((1)_{\alpha<\tau}-{\mathsf r})\sigma_\varphi\restriction_{\ell^q(\tau)}$,
\item is a $\psi-$derivation if and only if $\psi=\frac12\sigma_\varphi\restriction_{\ell^q(\tau)}$,
\item is not a (Jordan, Jordan triple) derivation.
\end{itemize}
Moreover if operator $d:\ell^p(\tau)\to\ell^p(\tau)$ is a $\sigma_\varphi\restriction_{\ell^p(\tau)}-$derivation, then $d=0$
\begin{lemma}\label{amin10}
Consider linear mappings $\psi,\lambda$,
if $\sigma_\varphi\restriction_{\ell^p(\tau)}:\ell^p(\tau)\to\ell^p(\tau)$ is a  
$(\psi,\lambda)-$derivation, then for all $\alpha,\beta<\tau$ we have
\[\left\{\begin{array}{lc} \pi_\alpha(\psi({\mathsf w}^\beta))=\pi_\alpha(\lambda({\mathsf w}^\beta))=0 & 
\varphi(\alpha)\neq\beta \\
\pi_\alpha(\psi({\mathsf w}^\beta))+\pi_\alpha(\lambda({\mathsf w}^\beta))=1 & \varphi(\alpha)=\beta \end{array}\right.\]
i.e., for ${\mathsf r}=(\pi_\alpha(\psi({\mathsf w}^{\varphi(\alpha)})))_{\alpha<\tau}$ we have
$\psi({\mathsf w}^\beta)=\mathsf{r}{\mathsf w}^{\varphi^{-1}(\beta)}$ and 
$\lambda({\mathsf w}^\beta)=((1)_{\alpha<\tau}-\mathsf{r}){\mathsf w}^{\varphi^{-1}(\beta)}$ 
(for all $\beta<\tau$).
\end{lemma}
\begin{proof}
Suppose $\sigma_\varphi\restriction_{\ell^p(\tau)}:\ell^p(\tau)\to\ell^p(\tau)$ is a $(\psi,\lambda)-$derivation.
Choose distinct $\theta,\mu<\tau$,
then we have (use $\mathsf{w}^\theta\mathsf{w}^\mu=(0)_{\alpha<\tau}$):   
\begin{eqnarray*}
(0)_{\alpha<\tau} & = & \sigma_\varphi((0)_{\alpha<\tau}) = \sigma_\varphi(\mathsf{w}^\theta\mathsf{w}^\mu) \\
& = & \sigma_\varphi(\mathsf{w}^\theta)\psi(\mathsf{w}^\mu)+\lambda(\mathsf{w}^\theta)\sigma_\varphi(\mathsf{w}^\mu) \\
& = & {\mathsf w}^{\varphi^{-1}(\theta)}\psi(\mathsf{w}^\mu)+\lambda(\mathsf{w}^\theta){\mathsf w}^{\varphi^{-1}(\mu)} \\
& = & (\chi_{\varphi^{-1}(\theta)}(\alpha)\pi_\alpha(\psi(\mathsf{w}^\mu))+\pi_\alpha(\lambda(\mathsf{w}^\theta))
\chi_{\varphi^{-1}(\mu)}(\alpha))_{\alpha<\tau}
\end{eqnarray*} 
hence
\[\forall\theta\neq\mu\:\: \forall\alpha<\tau\:\: \chi_{\varphi^{-1}(\theta)}(\alpha)\pi_\alpha(\psi(\mathsf{w}^\mu))+\pi_\alpha(\lambda(\mathsf{w}^\theta))
\chi_{\varphi^{-1}(\mu)}(\alpha)=0\]
in particular
\[\forall\alpha<\tau\:\:\forall\theta\neq\varphi(\alpha)\:\:
\chi_{\varphi^{-1}(\theta)}(\alpha)\pi_\alpha(\psi(\mathsf{w}^{\varphi(\alpha)}))+\pi_\alpha(\lambda(\mathsf{w}^\theta))
\chi_{\varphi^{-1}(\varphi(\alpha))}(\alpha)=0\]      
\[\forall\alpha<\tau\:\: \forall\mu\neq\varphi(\alpha)\:\: \chi_{\varphi^{-1}(\varphi(\alpha))}(\alpha)\pi_\alpha(\psi(\mathsf{w}^\mu))+\pi_\alpha(\lambda(\mathsf{w}^{\varphi(\alpha)}))
\chi_{\varphi^{-1}(\mu)}(\alpha)=0\]
i,e.,
\[\forall\alpha<\tau\:\:\forall\theta\neq\varphi(\alpha)\:\:
\pi_\alpha(\psi(\mathsf{w}^\theta))=\pi_\alpha(\lambda(\mathsf{w}^\theta))=0\]
Also:
\begin{eqnarray*}
(\forall x\in\ell^p(\tau)\:\: \sigma_\varphi(x^2)&=&\sigma_\varphi(x)\psi(x)+\lambda(x)\sigma_\varphi(x))) \\
& \Rightarrow & (\forall\beta<\tau\:\:
 \sigma_\varphi(({\mathsf w}^\beta)^2)=\sigma_\varphi({\mathsf w}^\beta)(\psi({\mathsf w}^\beta)+\lambda({\mathsf w}^\beta))) \\
 & \Rightarrow & (\forall\beta<\tau\:\: \sigma_\varphi({\mathsf w}^\beta)=\sigma_\varphi({\mathsf w}^\beta)(\psi({\mathsf w}^\beta)+\lambda({\mathsf w}^\beta))) \\
 & \Rightarrow & (\forall\beta<\tau\:\: {\mathsf w}^{\varphi^{-1}(\beta)}={\mathsf w}^{\varphi^{-1}(\beta)}(\psi({\mathsf w}^\beta)+\lambda({\mathsf w}^\beta))) \\
 & \Rightarrow & (\forall\beta<\tau\:\: \forall\alpha\in\varphi^{-1}(\beta) \:\:
 \pi_{\alpha}(\psi({\mathsf w}^\beta))+\pi_\alpha(\lambda({\mathsf w}^\beta))=1)
\end{eqnarray*}
\end{proof}
\begin{theorem}\label{amin20}
Consider operators $\psi,\lambda$, the generalized shift
$\sigma_\varphi\restriction_{\ell^p(\tau)}:\ell^p(\tau)\to\ell^p(\tau)$ is   a
$(\psi,\lambda)-$derivation if and only if:
\begin{center}
\begin{tabular}{cc} 
\begin{tabular}{c}
there exists $(r_\alpha)_{\alpha<\tau}\in{\mathbb C}^\tau$
with $\psi((x_\alpha)_{\alpha<\tau})=(r_\alpha x_{\varphi(\alpha)})_{\alpha<\tau}$ \\
and $\lambda((x_\alpha)_{\alpha<\tau})=((1-r_\alpha) x_{\varphi(\alpha)})_{\alpha<\tau}$
for all $(x_\alpha)_{\alpha<\tau}\in\ell^p(\tau)$  \\
 \end{tabular} & $(*)$ \\
\end{tabular}
\end{center}
\end{theorem}
\begin{proof}
Suppose $\sigma_\varphi\restriction_{\ell^p(\tau)}:\ell^p(\tau)\to\ell^p(\tau)$ is a  
$(\psi,\lambda)-$derivation, let 
\linebreak
${\mathsf r}:=(\pi_\alpha(\psi({\mathsf w}^{\varphi(\alpha)})))_{\alpha<\tau}$.
and $\mathcal{D}:=\{s_1{\mathsf w}^{\beta_1}+\cdots+s_n{\mathsf w}^{\beta_n}:n\geq1,\beta_1,\ldots,\beta_n\in\tau,
s_1,\ldots,s_n\in{\mathbb C}\}$, then $\mathcal D$ is a dense subset of $\ell^p(\tau)$ and for all
$\beta_1,\ldots,\beta_n\in\tau,s_1,\ldots,s_n\in{\mathbb C}$ we have (use Lemma~\ref{amin10}):
\begin{eqnarray*}
\psi(s_1{\mathsf w}^{\beta_1}+\cdots+s_n{\mathsf w}^{\beta_n}) & = & 
s_1\psi({\mathsf w}^{\beta_1})+\cdots+s_n\psi({\mathsf w}^{\beta_n}) \\
& = & s_1\mathsf{r}{\mathsf w}^{\varphi^{-1}(\beta_1)}+\cdots+s_n\mathsf{r}{\mathsf w}^{\varphi^{-1}(\beta_n)} \\
& = & s_1\mathsf{r}\sigma_\varphi({\mathsf w}^{\beta_1})+\cdots+s_n\mathsf{r}\sigma_\varphi({\mathsf w}^{\beta_n}) \\
& = & \mathsf{r}\sigma_\varphi(s_1{\mathsf w}^{\beta_1}+\cdots+s_n{\mathsf w}^{\beta_n})
\end{eqnarray*}
therefore $\psi(x)=\mathsf{r}\sigma_\varphi(x)$ for all $x\in\mathcal{D}$. Since $\mathcal D$ is a dense
subset of $\ell^p(\tau)$ equipped with pointwise convergence, and $\sigma_\varphi\restriction_{\ell^p(\tau)}$ and $\psi$
is continuous on $\ell^p(\tau)$ equipped with metric induced by $||\:||_p$, hence they are continuous
with pointwise convergence topology, we have $\psi(x)=\mathsf{r}\sigma_\varphi(x)$ for all $x\in\ell^p(\tau)$.
Using a similar method we have $\lambda(x)=((1)_{\alpha<\tau}-\mathsf{r})\sigma_\varphi(x)$ for all $x\in\mathcal{D}$
therefore $\lambda(x)=((1)_{\alpha<\tau}-\mathsf{r})\sigma_\varphi(x)$ for all $x\in\ell^p(\tau)$.
Hence $\mathsf{r},\psi,\lambda$ satisfy $(*)$.
\\
On the other hand if there exists $(r_\alpha)_{\alpha<\tau}=\mathsf{r}\in{\mathbb C}^\tau,\psi,\lambda$ satisfy $(*)$,
then for all $x,y\in\ell^p(\tau)$ we have:
\begin{eqnarray*}
\psi(x)\sigma_\varphi(y)+\sigma_\varphi(x)\lambda(y)
	& = & \mathsf{r}\sigma_\varphi(x)\sigma_\varphi(y)+\sigma_\varphi(x)((1)_{\alpha<\tau}-\mathsf{r})\sigma_\varphi(y) \\
& = & (\mathsf{r}+(1)_{\alpha<\tau}-\mathsf{r})\sigma_\varphi(x)\sigma_\varphi(y) \\
& = & (1)_{\alpha<\tau}\sigma_\varphi(x)\sigma_\varphi(y)=\sigma_\varphi(x)\sigma_\varphi(y)=\sigma_\varphi(xy)
\end{eqnarray*}
and $\sigma_\varphi\restriction_{\ell^p(\tau)}$ is a $(\psi,\lambda)-$derivation.
\end{proof}
\begin{corollary}\label{amin30}
For linear mapping $\psi$,
if $\sigma_\varphi\restriction_{\ell^p(\tau)}:\ell^p(\tau)\to\ell^p(\tau)$ is a  
$\psi-$derivation, then for all $\beta<\tau$ we have $\psi({\mathsf w}^\beta)=\frac12\sigma_\varphi({\mathsf w}^\beta)$.
\\
Thus for operator $\psi:\ell^p(\tau)\to\ell^p(\tau)$, $\sigma_\varphi\restriction_{\ell^p(\tau)}:\ell^p(\tau)\to\ell^p(\tau)$ is a  
$\psi-$derivation if and only if $\psi=\frac12\sigma_\varphi\restriction_{\ell^p(\tau)}$.
\end{corollary}
\begin{proof}
Use Lemma~\ref{amin10} and Theorem~\ref{amin20}.
\end{proof}
\begin{corollary}\label{amin40}
The generalized shift $\sigma_\varphi\restriction_{\ell^p(\tau)}:\ell^p(\tau)\to\ell^p(\tau)$ is not
a (Jordan, Jordan triple) derivation.
\end{corollary}
\begin{proof}
If $\sigma_\varphi\restriction_{\ell^p(\tau)}:\ell^p(\tau)\to\ell^p(\tau)$ is 
a (Jordan) derivation., then  for $\theta<\tau$ we have:
\begin{eqnarray*}
\{0,1\}^\tau\ni\mathsf{w}^{\varphi^{-1}(\theta)} & = & \sigma_\varphi(\mathsf{w}^{\theta}) =\sigma_\varphi((\mathsf{w}^{\theta})^2)\\
&= &2\mathsf{w}^\theta \sigma_\varphi(\mathsf{w}^{\theta})=2\mathsf{w}^{\varphi^{-1}(\theta)\cap\{\theta\}}\in\{0,2\}^\tau
\end{eqnarray*}
hence $\mathsf{w}^{\varphi^{-1}(\theta)}\in\{(0)_{\alpha<\tau}\}=\{0,1\}^\tau\cap\{0,2\}^\tau$ and $\varphi^{-1}(\theta)=
\varnothing$ for all $\theta<\tau$, 
which is a contradiction.
\end{proof}
\begin{lemma}\label{amin44}
If the linear mapping $d$,
is a $\sigma_\varphi\restriction_{\ell^p(\tau)}-$derivation, then $d(\mathsf{w}^\beta)=0$
for all $\beta<\tau$.
\end{lemma}
\begin{proof}
For $\beta<\tau$ we have:
\[d(\mathsf{w}^\beta)=d((\mathsf{w}^\beta)^2)=d(\mathsf{w}^\beta)\sigma_\varphi(\mathsf{w}^\beta)=2d(\mathsf{w}^\beta)
\mathsf{w}^{\varphi^{-1}(\beta)}\]
Hence $d(\mathsf{w}^\beta)((1)_{\alpha<\tau}-2\mathsf{w}^{\varphi^{-1}(\beta)})=0$
which leads to $d(\mathsf{w}^\beta)=0$
(note that
\linebreak 
$(1)_{\alpha<\tau}-2\mathsf{w}^{\varphi^{-1}(\beta)}\in(\mathbb{C}\setminus\{0\})^\tau$).
\end{proof}
\begin{theorem}\label{amin48}
If the operator $d$
is a $\sigma_\varphi\restriction_{\ell^p(\tau)}-$derivation, then $d=0$  (see \cite{zero} too).
\end{theorem}
\begin{proof}
Use Lemma~\ref{amin44}.
\end{proof}
\begin{note}
For $\mathsf{r}\in[0,1]^\tau$ one may consider
operators $\mathsf{r}\sigma_\varphi\restriction_{\ell^p(\tau)},
((1)_{\alpha<\tau}-\mathsf{r})\sigma_\varphi\restriction_{\ell^p(\tau)}:\ell^p(\tau)\to\ell^p(\tau)$ so by 
Theorem~\ref{amin20} for 
$\mathcal{E}:=\{(\psi,\lambda):\psi,\lambda:\ell^p(\tau)\to\ell^p(\tau)$ are operators such that
$\sigma_\varphi\restriction_{\ell^p(\tau)}$ is a $(\psi,\lambda)-$derivation$\}$ we have
\[2^{\aleph_0\tau}=card([0,1]^\tau)\leq card(\mathcal{E})\leq card(\mathbb{C}^\tau)=2^{\aleph_0\tau}\]
which leads to $card(\mathcal{E})=2^{\aleph_0\tau}=2^{\max(\aleph_0,\tau)}$,
\end{note}
\section{Is $\sigma_\varphi\restriction_{\ell^p(\tau)}$ a generalized derivation?}
\noindent Let's recall, in Banach algebra $\mathcal A$ \cite{survey, niknam}:
\begin{itemize}
\item we say linear mapping $D:{\mathcal A}\to{\mathcal A}$ is a
	generalized derivation if there exists a derivation $d:\mathcal{A}\to\mathcal{A}$
	such that $D(ab)=D(a)b+ad(b)$ for all $a,b\in\mathcal{A}$,
\item we say linear mapping $D:{\mathcal A}\to{\mathcal A}$ is a
	generalized Jordan derivation if there exists a derivation $d:\mathcal{A}\to\mathcal{A}$
	such that $D(a^2)=D(a)a+ad(a)$ for all $a\in\mathcal{A}$, , thus any generalized derivation is a
	generalized Jordan derivation,
\item we say linear mapping $D:{\mathcal A}\to{\mathcal A}$ is a
	generalized Jordan triple derivation if there exists a Jordan triple derivation $d:\mathcal{A}\to\mathcal{A}$
	such that $D(aba)=D(a)ba+ad(b)a+abd(a)$ for all $a,b\in\mathcal{A}$, 
\end{itemize}
In this section we prove that 
$\sigma_\varphi\restriction_{\ell^p(\tau)}:\ell^p(\tau)\to\ell^p(\tau)$ is a generalized (Jordan, Jordan triple) derivation if and only
if $\varphi=id_\tau$.
\begin{lemma}\label{amin50}
Consider derivation $d$.
If $\sigma_\varphi(x^2)=\sigma_\varphi(x)x+xd(x)$ for all
$x\in\ell^p(\tau)$, then $\varphi=id_\tau$ and $\sigma_\varphi\restriction_{\ell^p(\tau)}=id_{\ell^p(\tau)}$.
\end{lemma}
\begin{proof}
For $\beta<\tau$ we have:
\begin{eqnarray*}
{\mathsf w}^{\varphi^{-1}(\beta)} & = & \sigma_\varphi({\mathsf w}^\beta)=\sigma_\varphi(({\mathsf w}^\beta)^2) \\
& = & \sigma_\varphi({\mathsf w}^\beta){\mathsf w}^\beta+{\mathsf w}^\beta d({\mathsf w}^\beta) =
	{\mathsf w}^{\{\beta\}\cap \varphi^{-1}(\beta)}+{\mathsf w}^\beta d({\mathsf w}^\beta)
\end{eqnarray*}
hence $(0)_{\alpha<\tau}={\mathsf w}^\beta({\mathsf w}^{\varphi^{-1}(\beta)}-{\mathsf w}^{\{\beta\}\cap \varphi^{-1}(\beta)}-{\mathsf w}^\beta d({\mathsf w}^\beta))=-{\mathsf w}^\beta d({\mathsf w}^\beta)=\frac{-1}2d({\mathsf w}^\beta)$
which leads to $d({\mathsf w}^\beta)=0$ and ${\mathsf w}^{\varphi^{-1}(\beta)}={\mathsf w}^{\{\beta\}\cap \varphi^{-1}(\beta)}$, therefore:
\[\forall\beta<\tau\:\: \varphi^{-1}(\beta)=\{\beta\}\cap \varphi^{-1}(\beta)\]
in particular $\alpha\in \varphi^{-1}(\varphi(\alpha))=\{\varphi(\alpha)\}\cap \varphi^{-1}(\varphi(\alpha))$
and $\alpha=\varphi(\alpha)$ for all $\alpha<\tau$.
So $\varphi=id_\tau$.
\end{proof}
\begin{theorem}\label{amin60}
The generalized shift
$\sigma_\varphi\restriction_{\ell^p(\tau)}:\ell^p(\tau)\to\ell^p(\tau)$ is a  generalized (Jordan, Jordan triple) derivation if and only
if $\varphi=id_\tau$.
\end{theorem}
\begin{proof}
Use Lemma~\ref{amin50}.
\end{proof}
\begin{note}
In Banach Algebra $\mathcal A$,
we call the sequence of linear mappings 
	$\{d_n\}_{n\geq0}$ a higher derivation (resp. Jordan higher derivation, Jordan triple higher derivation) if $d_n(xy)=\mathop{\sum}\limits_{i+j=n}
	d_i(x)d_{j}(y)$ (rep. $d_n(x^2)=\mathop{\sum}\limits_{i+j=n}
	d_i(x)d_j(x)$, $d_n(xyx)=\mathop{\sum}\limits_{i+j+k=n}
	d_i(x)d_j(y)d_k(x)$)	
	for all $x,y\in\mathcal{A}$ and $n\geq0$  \cite{jordan, survey}.
Using Lemma~\ref{amin44} and Theorem~\ref{amin48} for sequence $\{d_n\}_{n\geq0}$ of operators:
\\
(1) If $\{d_n\}_{n\geq0}$ is a (Jordan, Jordan triple) higher derivation over $\ell^p(\tau)$ with $d_0=\sigma_\varphi\restriction_{\ell^p(\tau)}$,
then $d_1=d_2=\cdots=0$.
\end{note}

\vspace{3mm}
\noindent{\small {\bf Safoura Arzanesh},
Faculty of Mathematics, Statistics and Computer Science,
College of Science, University of Tehran,
Enghelab Ave., Tehran, Iran
\\
({\it e-mail}: arzanesh.parsian@gmail.com)}
\vspace{3mm}
\\
{\small {\bf Fatemah Ayatollah Zadeh Shirazi},
Faculty of Mathematics, Statistics and Computer Science,
College of Science, University of Tehran,
Enghelab Ave., Tehran, Iran
\\
({\it e-mail}: f.a.z.shirazi@ut.ac.ir  ,  fatemah@khayam.ut.ac.ir)}
\vspace{3mm}
\\
{\small {\bf Arezoo Hosseini},
Faculty of Mathematics, College of Science, Farhangian University, Pardis Nasibe--shahid sherafat, Enghelab Ave., Tehran, Iran
\\
({\it e-mail}: a.hosseini@cfu.ac.ir)}

\begin{thebibliography}{99}

\bibitem{jordan} B. Arslan, H. Inceboz, \textit{A characterization of generalized Jordan derivations on Banach algebras}, Period. Math. Hungar., 69, no. 2 (2014), 139--148.

\bibitem{compact} F. Ayatollah Zadeh Shirazi, F. Ebrahimifar, \textit{Is there any nontrivial compact generalized shift operator on Hilbert spaces?}, Rend. Circ. Mat. Palermo (2), 68, no. 3 (2019), 453--458.

\bibitem{note} F. Ayatollah Zadeh Shirazi, N. Karami Kabir, F. Heidari Ardi,
	\textit{A Note on shift theory}, Mathematica Pannonica, Proceedings of ITES-2007,
19/2 (2008), 187--195.

\bibitem{dev} F. Ayatollah Zadeh Shirazi, J. Nazarian Sarkooh, B. Taherkhani,
	\textit{On Devaney chaotic generalized shift dynamical systems}, Studia Scientiarum Mathematicarum Hungarica, 50, no. 4 (2013), 509--522.

\bibitem{taher} F. Ayatollah Zadeh Shirazi, F. Ebrahimifar, B. Taherkhani, \textit{Semi--distality and related topics in generalized shift dynamical systems}, Iran. J. Sci. Technol. Trans. A Sci., 41, no. 4 (2017), 957--963.

\bibitem{folland} G. B. Folland, \textit{Real analysis, Modern techniques and their applications}, 2nd ed.,  A Wiley--Interscience Publication, (1999).

\bibitem{anna} A. Giordano Bruno, \textit{Algebraic entropy of generalized shifts on direct products}, Comm. Algebra, 38, no. 11 (2010), 4155--4174.

\bibitem{survey} C. Haetinger, M. Ashraf, Sh. Ali, \textit{On higher derivations: a survey}, Int. J. Math. Game Theory Algebra 19, no. 5-6 (2011), 359--379.

\bibitem{zero} A. Hosseini, \textit{Some conditions under which derivations are zero on Banach $\ast$-algebras}, Acta Univ. Sapientiae Math., 9, no. 1 (2017), 176--184. 

\bibitem{niknam} A. Hosseini, M. Hassani, A. Niknam, \textit{Some results on generalized $\sigma$-derivations}, J. Adv. Res. Pure Math., 5, no. 3 (2013), 61--72. 

\bibitem{meng} Ch. Hou, Q. Meng, \textit{Continuity of $(\alpha,\beta)$-derivations of operator algebras}, J. Korean Math. Soc. 48, no. 4 (2011), 823--835.

\bibitem{kh} A, Khotanloo, G. H. Esslamzadeh, B. Tabatabaie Shourijeh, \textit{A unified approach to various notions of derivation}, Iran. J. Sci. Technol. Trans. A Sci., 43, no. 5 (2019), 2551--2557.

\bibitem{nik} I. Nikoufar, \textit{On generalized notions of derivations}, Iran. J. Sci. Technol. Trans. A Sci., 43, no. 5 (2019), 2607--2611.

\bibitem{sakai} S. Sakai, \textit{Derivations of $W^\ast$-algebras. Ann. of Math.}, 2 / 83 (1966), 273--279. 

\end{thebibliography}
\end{document}